\documentclass[11pt]{article}

\usepackage{versions}
\includeversion{T^R-T/2}
\excludeversion{T^R-T^G}
\excludeversion{indicatorfunctions}

\usepackage{amsmath,amsfonts,amssymb,amstext,amsthm,bbm} 
\usepackage{bm} 
\numberwithin{equation}{section}

\theoremstyle{plain}
\newtheorem{theorem}{Theorem}[section]
\newtheorem{lemma}[theorem]{Lemma}
\newtheorem{corollary}[theorem]{Corollary}
\theoremstyle{definition}
\newtheorem{definition}[theorem]{Definition}
\newtheorem{example}[theorem]{Example}
\newtheorem{remark}[theorem]{Remark}

\usepackage{mathpazo}   

\newcommand{\sqln}[1]{\sqrt{#1 \ln #1}}

\newcommand{\zmin}{t_{\min}}

\newcommand\range[2]{\in\{#1,\dots,#2\}}
			
\newcommand{\ignore}[1]{}
\title{Concentration Inequalities for Random Sets}

\newcommand{\abs}[1]{\left|#1\right|}
\newcommand{\vcdim}{\text{VCDim}}
\newcommand{\rademacher}{\text{Rad}}
\newcommand{\uidim}{\text{UIDim}}

\usepackage{graphicx,url}

\author{
	Erel Segal-Halevi\footnote{erelsgl@gmail.com}
	~~and~
	Avinatan Hassidim\footnote{avinatanh@gmail.com}
	\\
	Bar-Ilan University, Ramat-Gan 5290002, Israel
}

\begin{document}
\maketitle
	
\begin{abstract}
In a large, possibly infinite population, each subject is colored red with probability $p$, independently of the others. Then, a finite sub-population is selected, possibly as a function of the coloring. The imbalance in the sub-population is defined as the difference between the number of reds in it and p times its size. This paper presents high-probability upper bounds (tail-bounds) on this imbalance. To present the upper bounds we define the *UI dimension* --- a new measure for the richness of a set-family. We present three simple rules for upper-bounding the UI dimension of a set-family. Our upper bounds on the imbalance in a sub-population depend only on the size of the sub-population and on the UI dimension of its support. We relate our results to known concepts from machine learning, particularly the VC dimension and Rademacher complexity.
\end{abstract}

\section{Introduction}
In many experimental processes, a sample is randomly taken from a population, a certain measurement is done on the sample and then applied to the entire population. The measurement done on the sample may not be entirely accurate due to the imbalance caused by the random sampling process. It is desired to have an upper bound on this imbalance. We model this process in the following way.

There is a population $O$ with a large number of subjects. We do not assume any bound on the size of $O$ (in particular, $O$ may be finite or infinite). The population is colored randomly: each subject is colored red with probability $p>0$ and remains uncolored with probability $1-p$. Then, a finite sub-population containing $t$ subjects is selected; denote this sub-population by $T$ (so $T\subseteq O$ and $|T|=t$). Denote by $T^R$ the set of red subjects in $T$. The difference $\abs{|T^R|-p\cdot |T|}$ denotes the imbalance caused by the randomization process. What is high-probability upper bound on this imbalance, as a function of $t$ and $p$?

There are two extreme cases:
\begin{itemize}
\item The easy case is when $T$ does not depend on the coloring, i.e, $T$ is a deterministic set defined before the coin-tosses. Then, the expected value of $|T^R|$ is $p\cdot|T|$. The difference between them can be bounded using standard concentration inequalities, e.g, with probability $1-O(1/t)$, the imbalance is $O(\sqln{t})$. See below Lemma \ref{lem:half-to-half}.
\item The hard case is when $T$ can depend on the coloring in an arbitrary way. Then, no non-trivial upper bound exists. For example, an adversary can select $T$ to contain $t$ red subjects. In this case, $T^R=T$ and the imbalance is $\Theta(t)$.
\end{itemize}
We are interested in an intermediate case, in which $T$ may depend on the coloring but only in a restricted way. As an example, suppose all the subjects in $O$ are placed on the real line, and $T$ must be an interval. $T$ may depend on the coloring, so it is a random variable and the standard concentration inequalities do not apply. However, the restriction to an interval means that an adversary cannot always select $t$ red subjects. Therefore we may hope to have a non-trivial upper bound on the imbalance $\abs{|T^R|-p\cdot |T|}$.
Our goal in this paper is to define a family of random sets and prove high-probability upper bounds on their imbalance.

Our motivating application comes from economics. Often, to determine a price for an item, a market-research is conducted in which a random sample of the buyer population is used to calculate an `optimal' price. Naturally, a price that is optimal in the sample might not be optimal in the global population. The optimality of the price depends on the set of buyers who want to buy the item in that price. Denote this set by $T$. Since the price depends on the sampling, it is a random variable, so $T$ is a random variable too. However, it is reasonable to assume that $T$ is an interval, since it includes all buyers whose valuation for the item is more than the price. The concentration bounds we develop in the present paper can be used to bound the imbalance in $T$.

\section{Deterministic-set Sampling Lemma}
As a baseline, we state a known lemma for deterministic sets. We prove it in two variants that will be useful later.

Below, the shorthand "w.p. $x$" means "with probability of at least $x$".
	
\begin{T^R-T/2}
\begin{lemma}[Deterministic-set Sampling Lemma] \label{lem:half-to-half}
	If $T$ is a deterministic set, then for every constant $r\geq 1$:
	\begin{align}
	\label{eq:zeq}
	\text{If $|T|=t$:}&& \text{w.p. } 1-{2\over t^{2 r^2}}: && \abs{|T^R|-{p\cdot|T|}} &< r \sqln{t}
	\\
	\label{eq:zmin}
	\text{If $|T|\geq \zmin$:}&& \text{w.p. } 1-{2\over (\zmin)^{2 r^2}}: && \abs{|T^R|-{p\cdot|T|}} &< r \sqln{|T|}
	\end{align}
\end{lemma}
\begin{proof}
		For every subject in $T$, define a random variable that equals $1$ if the subject is red and $0$ otherwise. These are i.i.d. random variables each of which is bounded in $[0,1]$. The sum of these variables is $|T^R|$ and the expectation of the sum is $p\cdot|T|$. For every $q\geq 0$, define the \emph{failure probability} as:
		\begin{align*}
		P_{fail,q}: = \Pr\left[\abs{|T^R|-{p\cdot|T|}}>q\right]
		\end{align*}
		By Hoeffding's inequality:
		\begin{align*}
		P_{fail,q} < 2\exp\left(\frac{-2 q^2}{\sum_T{(1-0)^2}}\right) \leq 2\exp\left(\frac{-2 q^2}{1\cdot |T|} \right)
		\end{align*}
		
		To get (\ref{eq:zeq}), let $q=r \sqln{t}$; then  $P_{fail,q}\leq 2/t^{2 r^2}$.
		
		To get (\ref{eq:zmin}), let $q=r\sqln{|T|}$; then $P_{fail,q} \leq 2/|T|^{2 r^2} \leq 2/(\zmin)^{2 r^2}$.
\end{proof}
\end{T^R-T/2}

\section{$d$-bounded random-sets}
If the set $T$ is not deterministic but depends on the outcomes of the random sampling, then Lemma \ref{lem:half-to-half} is not true without further restrictions. To handle such cases in a meaningful way we impose some structure on the possible values of the set $T$.

	\begin{definition}\label{def:random-set}
		A \textbf{random-set} is a random variable whose possible values are subsets of the global population $O$, and whose value depends on the random coloring process. The \emph{support} of a random-set is the collection of sets that it can equal with positive probability.
	\end{definition}
	\begin{definition}\label{def:d-bounded-set}
		Given an integer $d\geq 1$, a set family $H$ is called \textbf{$d$-bounded} if for every integer $j\geq 1$, the number of elements in $H$ with cardinality $j$ is at most $(j+1)^{d-1}$.
	\end{definition}
	\begin{definition}\label{def:d-bounded}
		Given an integer $d\geq 1$, a random-set $T$ is called \textbf{$d$-bounded} if its support is a $d$-bounded set-family.
	\end{definition}
	\begin{example}\label{exm:1-bounded}
		Let $O$ be a finite set of real numbers. Let $p$ be some real-valued random variable. Define $T = \{o\in O| o<p\}$. 
		$T$ is a random-set, since its value is a set that depends on a random variable. It is 1-bounded, because for every integer $j$, there is at most one possible outcome of $T$ with cardinality $j$ --- it is the set of $j$ smallest numbers in $O$. We will later generalize this example and show how to construct $d$-bounded random-sets. \qed
	\end{example}
	A $d$-bounded random-set is useful because of the following lemma.

\begin{T^R-T/2}
	\begin{lemma}[Random-set Sampling Lemma] \label{lem:half-to-half-rnd2}
		Let $T$ be a $d$-bounded random-set, for some integer $d\geq 1$. Then:
		\begin{align}
		\label{eq:zmin-rnd2}
		\text{w.p.  $1-4/\zmin$}:&&\text{If $|T|\geq \zmin$: } && \abs{|T^R|-{p\cdot|T|}}< d\cdot \sqln{|T|}
		\end{align}
	\end{lemma}
	\begin{proof}
		Denote the support of $T$ by $H$ (it is a set-family). Denote the subset of $H$ containing sets of $j$ elements by $H^j$. Every set $h_{j}\in H^j$ is deterministic, so it is eligible to the Deterministic-set Sampling Lemma. Substituting $r=d$ in (\ref{eq:zeq}) gives, for every $j,h_j$:
		\begin{align}\label{eq:halving-wjt}
		&& \text{w.p. } 1-{2\over j^{2 d^2}}: &&  \abs{h_{j}^R - {p\cdot j}} < d\cdot \sqln{j}
		\end{align}
		Since $H$ is $d$-bounded, the number of different sets in $H^j$ is at most $(j+1)^{d-1}$. Hence, by the union bound,  the above statement is true for \emph{all} sets in $H^j$ w.p. $1-2 (j+1)^{d-1}/j^{2d^2} \geq 1-4/j^2$:
		\begin{align}\label{eq:halving-wj}
		&& \text{w.p. } 1-{4\over j^{2}}: && \forall h_j\in H^j: && \abs{h_{j}^R - {p\cdot j}} < d\cdot \sqln{j}
		\end{align}
		
		Using the union bound again, the probability that inequality (\ref{eq:halving-wj}) is false for at least one $j\geq \zmin$ is upper-bounded by:
		\begin{align*}
		\sum_{j=\zmin}^\infty\frac{4}{j^2} \approx \int_{x=\zmin}^\infty \frac{4}{x^2}dx = \frac{4}{\zmin}
		\end{align*}
		so w.p. $1-4/\zmin$, inequality (\ref{eq:halving-wj}) is true for all $h_j$ with $|h_j|\geq \zmin$. This implies (\ref{eq:zmin-rnd2}).
	\end{proof}
\end{T^R-T/2}

Motivated by the Random-set Sampling Lemma, we now present ways to construct $d$-bounded random-sets.
	
\section{The UI-dimension of random-sets}
The property of being $d$-bounded is not preserved under set operations such as union and intersection. Below, we define a stronger property that is preserved. We call it the \emph{UI dimension} since it is preserved under Union and Intersection.\footnote{It has some similarities to the \emph{VC dimension} from learning theory, but they are not identical. See below Section \ref{sec:learning}.} We need some notation. For a set-family $H$ and a set $h'$:
\begin{align*}
H\cap h' := \{h \cap h' | h\in H\}
\\
H\setminus h' := \{h \setminus  h' | h\in H\}
\end{align*}

\begin{definition}\label{def:d-dinensional-set}
	The \textbf{UI dimension} of a set-family $H$, denoted $\uidim(H)$,  is the smallest integer $d$ such that, for every set $h'$, the family $H\cap h'$ is $d$-bounded (as defined in Definition \ref{def:d-bounded}).
	
\end{definition}
\begin{definition}\label{def:d-dimensional}
	The \textbf{UI dimension} of a random-set $T$, denoted $\uidim(T)$, is the UI dimension of the support of $T$.
\end{definition}
Obviously, a random-set with UI-dimension $d$ is also $d$-bounded, so it is eligible for the Random-Set Sampling Lemma (\ref{lem:half-to-half-rnd2}).
	
Below we provide three rules for constructing random sets with a bounded UI dimension. The first one is the \emph{Containment-Order Rule}.

\begin{definition}\label{def:filtration}
	A set-family is called \textbf{ordered-by-containment} if the sets in the family can be indexed $\{h_1,h_2,\dots\}$ such that 		for all $i<j$: $h_i\subset h_j$.
\end{definition}

\begin{remark}
	In measure theory and stochastic processes theory, a set-family that is ordered-by-containment is called a \emph{filtration}.
\end{remark}

\begin{lemma}
	\label{lem:1-dimensional}
	Let $H$ be a family of finite sets. Then $H$ is ordered-by-containment, if-and-only-if its UI dimension is at most 1: $\uidim(H)\leq 1$.
\end{lemma}
\begin{proof}
	$\Rightarrow:$ If $H$ is ordered-by-containment, then for every set $h'$, $H\cap h'$ is clearly also ordered-by-containment.
	In every set-family that is ordered-by-containment, $\forall i,j: i<j$: $h_i\subset h_j$. The sets $h_i,h_j$ are finite, so there can be at most a single $h_i$ with any given cardinality. Hence, for every $h'$, the family $H\cap h'$ is 1-bounded. Hence, $\uidim(H)$ is at most 1.
	
	$\Leftarrow:$ If $H$ is not ordered-by-containment, then it contains two sets, $h_1$ and $h_2$, that are incomparable in terms of containment (no one contains the other).  Then there are two elements,  $x_1 \in h_1\setminus h_2$  and  $x_2 \in h_2\setminus h_1$. 
	Now, consider the intersection of $H$ with the doubleton $h':=\{x_1,x_2\}$.   This is a set-family that contains two different sets with cardinality 1.   So $H\cap h'$ is not 1-bounded.  So $\uidim(H)>1$.
\end{proof}

\begin{corollary}[{Containment-Order Rule}]
	If the support of a random-set $T$ is ordered-by-containment, then $\uidim(T)\leq 1$.
\end{corollary}

\begin{example}\label{exm:filtration}
	Consider a family $\{h_1,h_2,\dots\}$ where for every $j$, $h_j$ is the set of $j$ smallest elements in a finite population $O$ of real numbers. This family is clearly ordered-by-containment. By the Containment-Order Rule, the random-set of Example \ref{exm:1-bounded} has a UI dimension of at most 1.
	\qed
\end{example}

\section{Intersections and unions of random-sets}
	\begin{lemma} (Single-Set Intersection)
		\label{lem:intersect-1}
		For every set-family $H$ and every set $h'$:
		\begin{align*}
		\uidim(H\cap h') \leq \uidim(H)
		\end{align*}
	\end{lemma}
	\begin{proof}
		Let $d = \uidim(H\cap h')$.
		We have to prove that for any set $h''$, the set-family $(H\cap h')\cap h''$ is $d$-bounded. Indeed, $(H\cap h')\cap h'' = H\cap(h'\cap h'')$, and because $\uidim(H)d$, by definition $H\cap(h'\cap h'')$ is $d$-bounded.
	\end{proof}
	\begin{corollary}\label{cor:intersect-1}
		The intersection of a random set with a deterministic set yields a random set with a weakly smaller UI-dimension.
	\end{corollary}
	
	\begin{lemma}[{Union Rule}]
		\label{lem:union}
		For any sequence of random sets $T_1,\ldots,T_n$:
		\begin{align*}
		\uidim(T_1\cup \cdots \cup T_n) \leq \uidim(T_1)+\ldots+\uidim(T_n)
		\end{align*}
	\end{lemma}
	\begin{proof}
		It is sufficient to prove the lemma for $n=2$; the proof for any $n$ follows by induction.
		
		For each $i\in\{1,2\}$, let $H_i$ be the support of $T_i$ and let $d_i$ be its UI-dimension. Let $H$ be the support of the union $T_1\cup T_2$.
		Let $h'$ by any deterministic set. We have to prove that $H\cap h'$ is a $(d_1+d_2)$-bounded set-family.
		
		By Lemma \ref{lem:intersect-1}, $\uidim(H_i\cap h')\leq \uidim(H_i) = d_i$. Suppose we want to construct a set in the family $H\cap h'$, and we want it to have cardinality $j$. The choices we can make are as follows:
		\begin{itemize}
			\item First, we choose a set $h_1$ from the family $H_1\cap h'$. The size of $h_1$ must be between $0$ and $j$, so we have at most $(j+1)$ choices for the size of $h_1$ and then at most $(j+1)^{d_1-1}$ for the set $h_1$ itself (because $H_1\cap h'$ is $d_1$-bounded). 
			\item Next, we choose a set $h_2$ from the family $(H_2\cap h')\setminus h_1$. The size of $h_2$ must be exactly $j-|h_1|$. Since the family $(H_2\cap h')\setminus h_1$ is $d_2$-bounded, we have at most $(j+1)^{d_2-1}$ choices for $h_2$.
		\end{itemize}
		All in all, the number of choices is at most $(j+1)\cdot (j+1)^{d_1-1}\cdot (j+1)^{d_2-1} = (j+1)^{d_1+d_2-1}$.
		
		Hence the set-family $H\cap h'$ is $(d_1+d_2)$-bounded. Since this is true for every set $h'$, $\uidim(H)\leq d_1+d_2$. Hence, $\uidim(T_1\cup T_2)\leq \uidim(T_1)+\uidim(T_2)$.
	\end{proof}
	\begin{example}\label{exm:2-dimensional}
		Let $O$ be a finite set of points in the plane, $O\subseteq \mathbb{R}^2$.
		
		Let $T := \{(x,y)\in O| x>p_x \text{ or } y>p_y\}$, where $p_x$ and $p_y$ are random variables.
	 $T$ is a union of the two random-sets: $T_x := \{(x,y)|x>p_x\}$ and $T_y := \{(x,y)|y>p_y\}$, whose UI dimension is at most 1 by the Order-Containment Rule.  
	 Therefore, $\uidim(T)\leq 2$ by the Union Rule.\qed
	\end{example}
	
\begin{figure}
	\begin{center}
		\includegraphics[scale=.8]{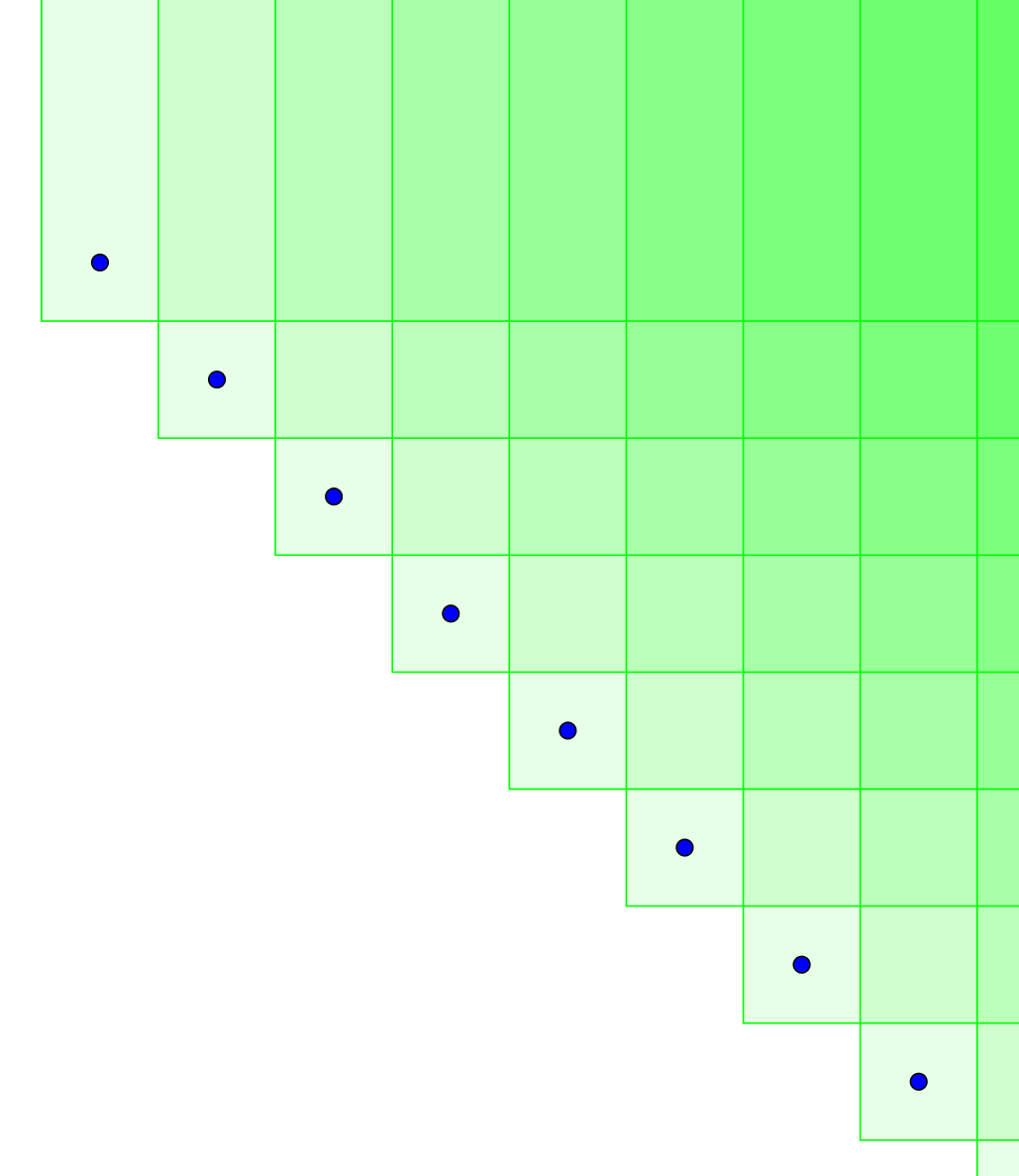}
	\end{center}
	\caption{
		\label{fig:intersection-unbounded} An intersection of two 1-dimensional random-sets may not be $d$-bounded.
	}
\end{figure}
The analogue of Lemma \ref{lem:union} for intersections of random-sets is not true.
\begin{example}\label{exm:2-dimensional-intersection}
Let $T = \{(x,y)\in O| x>p_x \text{ and } y>p_y\}$, where $p_x$ and $p_y$ are random variables. Then, $T = T_x\cap T_y$ are defined as in the previous example and $\uidim(T_x)=\uidim(T_y)=1$. However, $T$ is not $d$-bounded for any finite $d$. This is illustrated in Figure \ref{fig:intersection-unbounded}. The points represent the elements of $O$. Each quarter-plane represents a possible value of $T$. The cardinality of each such value is 1. Therefore, the number of sets of cardinality 1 in the support of $T$ can be as high as $|O|$ (the size of the global population). This is not bounded by $(1+1)^{d-1}$ for and constant $d$, since $|O|$ can be arbitrarily large. Similarly, for every $j\geq 1$, the number of sets of cardinality $j$ in the support of $T$ is not bounded by $(j+1)^{d-1}$ for any constant $d$.

Moreover, if $|O|$ is sufficiently large, with high probability there will be $t$ adjacent subjects colored red. An adversary can select a quarter-plane that contains all and only these red subjects. This quarter-plane will have the worst possible imbalance --- $t$.
\qed
\end{example}
Intersections of random-sets have a bounded UI dimension if one of the elements in the intersection has a bounded cardinality.
\begin{lemma}[{Intersection Rule}]		
	\label{lem:intersection}

For any sequence of random sets $T_0,T_1,\ldots,T_n$, if:
\begin{align*}
|T_0|<k \text{ w.p. 1,}
\end{align*}
then:
\begin{align*}
\uidim(T_0\cap \cdots \cap T_n) \leq (\uidim(T_0)+\ldots+\uidim(T_n))\cdot \lg{(k)}
\end{align*}
\end{lemma}
\begin{proof}
	For each $i\range{0}{n}$ let $H_i$ be the support of $T_i$ and let $d_i$ be its dimension. Let $H$ be the support of the intersection $T_0\cap \cdots \cap T_n$. Let $h'$ be any set. We have to prove that the set-family $H\cap h'$ is $((d_0+\ldots+d_n)\lg k)$-bounded.
	
	Every set $h\in H\cap h'$ with cardinality $j$ can be constructed as follows:
	\begin{itemize}
		\item Select a set $h_0\in H_{0}\cap h'$, having $j_0$ items.
		\item Select a set $h_1 \in (H_1\cap h_0)\cap h'$, having $j_1$ items;
		\item Select a set $h_2 \in ((H_2 \cap h_1) \cap h_0)\cap h'$, having $j_2$ items;	
		\item ... Select a set $h_{n} \in (H_{n}\cap ... \cap h_0)\cap h'$, having $j_{n}$ items, where $j_n=j$.
	\end{itemize}
	By assumption, $|T_0|<k$, so  $j_0\leq k-1$. Since $\uidim(H_0)=d_0$, given $j_0$, the number of choices for $h_0$ is at most $(j_0+1)^{d_0-1}\leq k^{d_0-1}$. Since there are at most $k$ choices for $j_0$, the total number of choices for $h_0$ is at most $k^{d_0}$.

	The set-families used in each of the following steps are intersections of a set-family with UI-dimension $d_i$ with  constant sets. Hence by Lemma \ref{lem:intersect-1} their UI dimension is at most $d_i$.
	In step $i$, there are at most $k$ choices for the number $j_i$, except the last step where the choice is determined in advance ($j_n$ must be equal to $j$). For every selection of $j_i$, there are at most $(j_i+1)^{d_i-1}\leq k^{d_i-1}$ choices for $h_i$. The total number of choices for $h_i$ is thus at most $k^{d_i}$ for $i\leq n-1$, and at most $k^{d_n-1}$ for $i=n$.
	
	Multiplying all numbers of choices gives that the total number of ways to construct $h$ is at most $k^{d_0+d_1+\ldots+d_n-1} \leq (j+1)^{(d_0+d_1+\ldots+d_n)\lg k - 1}$. Therefore $H\cap h'$ is $((d_0+d_1+\ldots+d_n)\lg k)$-bounded.
	
	Since this is true for every set $h'$, the set-family $H$ has a UI dimension of at most $((d_0+d_1+\ldots+d_n)\lg k)$.
\end{proof}
\begin{remark}
	If the set $T_0$ is deterministic and $|T_0|=k$, then the set-family $H_0$ is a singleton and there is only one way to choose $h_0$. Since $1=k^0$, the proof is still valid if we take $d_0=0$, so the resulting random-set has a UI dimension of at most $((d_1+\ldots+d_n)\lg k)$.
	
	Effectively, in the Intersection Rule, a deterministic set is equivalent to a zero-dimensional random-set. The same is true in the Union Rule.
\end{remark}

\section{Related Work in Learning Theory}
\label{sec:learning}
While our work is presented in general probabilistic terms, it is related to results in learning theory. We are grateful to Aryeh Kontorovich for telling us about this relation.\footnote{In this MathOverflow thread: http://mathoverflow.net/a/257050/34461} 

\subsection{UI Dimension and VC Dimension}
Our UI dimension is related to the well-known Vapnik-Chervonenkis (VC) dimension from learning theory. Both dimensions measure the complexity/richness of a set-family $H$. Both are related to the set-families of the form  $H\cap h' := \{h\cap h' | h\in H\}$. But they are different:
\begin{itemize}
\item The VC dimension is the largest integer $D$ such that, there exists a set $h'$ with cardinality $D$ for which $|H\cap h'|=2^{|h'|}$. Equivalently, it is the smallest integer such that, for all sets $h'$ with cardinality more than $D$, $|H\cap h'|<2^{|h'|}$.
\item The UI dimension is the smallest integer $d$ such that, for every set $h'$, the set-family $H\cap h'$ is $d$-bounded, i.e, for every $j$, $H\cap h'$ contains at most $(j+1)^{d-1}$ sets of cardinality $j$.
\end{itemize}
One connection between these two dimensions is given by the following lemma:
\begin{lemma}
	\label{lem:ui-vc}
If the UI dimension of $H$ is finite, then its VC dimension is finite too.
\end{lemma}
\begin{proof}
Let $d=\uidim(H)$. Then for every set $h'$, $H\cap h'$ contains at most $(j+1)^{d-1}$ sets of cardinality $j$. So the cardinality of $H\cap h'$ is bounded as:
\begin{align*}
|H\cap h'| \leq \sum_{j=0}^{|h'|} (j+1)^{d-1}
\end{align*}
which grows polynomially with $|h'|$. Therefore, when $|h'|$ is sufficiently large, $|H\cap h'| < 2^{|h'|}$. So $\vcdim(H)$ is finite.
\end{proof}
The opposite of Lemma \ref{lem:ui-vc} is not true: a set may have an infinite UI dimension and a finite VC dimension.

\begin{example}
	\label{exm:vc-vs-ui}
Let $H$ be a set-family whose members are half-lines of the form:
\begin{align*}
\{(x,y_0)\in \mathbb{R}^2 | x > x_0\}
\end{align*}
where $x_0,y_0$ are real numbers. Then:
\begin{itemize}
	\item $\vcdim(H)\leq 1$. \emph{Proof:} let $h'$ be a set of two points in $\mathbb{R}^2$. If the two points in $h'$ have the same $y$ coordinate, then $H\cap h'$ does not contain a singleton with the point whose $x$ coordinate is smaller. If the two points in $h'$ have different $y$ coordinate, then $H\cap h'$ does not contain $h'$. In both cases, $|H\cap h'|\leq 3 < 2^2$.
	\item $\uidim(H)=\infty$. \emph{Proof:} let $h'$ be the line $y=8$. Then, $H\cap h'$ contains infinitely many singletons.\qed
\end{itemize}
\end{example}
The above example also shows that one of the directions of Lemma \ref{lem:1-dimensional} is not true for the VC dimension. Specifically, if $H$ is ordered-by-containment then $\vcdim(H)=1$, but the opposite is not true: the set-family in Example \ref{exm:vc-vs-ui} has $\vcdim(H)=1$ and it is not ordered by containment.

The above results show that, in a sense, the UI dimension is stronger than the VC dimension, and it is interesting that it can be upper-bounded by rules related to simple set operations.

\subsection{Random-Set Sampling Lemma and Rademacher Complexity}
For the purpose of the present section, let $p=1/2$. Then, the expression that is upper-bounded by our Random-Set Sampling Lemma (\ref{lem:half-to-half-rnd2}) is:
\begin{align*}
\abs{|T^R|-p\cdot |T|} = {1\over 2}\abs{|T^R|-|T^U|},
\end{align*}
where $T^U$ is the set of uncolored elements in $T$. By symmetry considerations, for every failure probability $p$, in order to prove an upper bound with probability $1-2p$ on $\abs{|T^R|-|T^U|}$, it is sufficient to prove an upper bound with probability $1-p$ on:
\begin{align*}
|T^R|-|T^U|
\end{align*}
The latter expression can be presented as follows. For every subject in the global population $i\in O$, let $\sigma_i$ be a random variable drawn from the \emph{Rademacher distribution} --- a variable that equals $1$ with probability 1/2 and $-1$ otherwise, such that the $\sigma_i$ are independent. Then:
\begin{align*}
|T^R|-|T^U| = \sum_{i\in T} \sigma_i
\end{align*}
\begin{indicatorfunctions}
This can be written using the \emph{indicator function} of $T$ --- a function $f_T$ on $O$ that equals 1 for $i\in T$ and 0 otherwise:
\begin{align*}
|T^R|-|T^U| = \sum_{i\in O} \sigma_i\cdot f_T(i)
\end{align*}
The support of $T$ can be represented by a family of such indicator functions, $F_T$. We are interested in bounding the worst-case imbalance, i.e, 
\begin{align*}
\sup_{f\in F_T}\abs{|T^R|-|T^U|} = \sup_{f\in F_T}\abs{\sum_{i\in O} \sigma_i\cdot f(i)}
\end{align*}
\end{indicatorfunctions}%
Let $H$ be the support of $T$. We can represent each set $h\in H$ as a binary vector with $|O|$ elements, where for subject $i\in O$, $h_i=1$ if $i\in h$ and $h_i=0$ if $i\notin h$. So $H$ is a set of vectors in $\mathbb{R}^{|O|}$. We are interested in bounding the  worst-case imbalance of $T$ over its entire support, which is,
\begin{align*}
\sup_{T} {(|T^R|-|T^U|)}
=
\sup_{h\in H}{\sum_{i\in O} \sigma_i\cdot h_i}
\end{align*}
By McDiarmid's inequality, the latter expression is concentrated about its mean, so to get a high-probability upper-bound, it is sufficient to bound its expected value (expectation taken over the $\sigma_i$):
\begin{align*}
E_{\sigma}\bigg[\sup_{h\in H}{\sum_{i\in O} \sigma_i\cdot h_i}\bigg]
\end{align*}
The latter expression is equal, up to normalization, to the \emph{Rademacher complexity} of the set $H$:
\begin{align*}
E_{\sigma}\bigg[\sup_{h\in H}{\sum_{i\in O} \sigma_i\cdot h_i}\bigg]
=
m\cdot \rademacher(H)
\end{align*}
where $\rademacher(H)$ is the Rademacher complexity of $H$, and $m=|O|=$ the number of subjects in the global population.

To summarize, in order to calculate an upper bound on the imbalance $\abs{|T^R|-p\cdot |T|}$, it is sufficient to prove an upper bound on the Rademacher complexity of $H$ --- the support of $T$ --- where $H$ is considered as a set of binary vectors in $\{0,1\}^{m}$.

This raises the question of what known upper-bounds on $\rademacher(H)$ are relevant to the sets $H$ studied in the present paper. One of the common bounds is based on the VC dimension \cite[page 342]{ShalevShwartz2014Understanding}. If $\vcdim(H)=D$, then:
\begin{align*}
 \rademacher(H) &\leq \sqrt{{2 D \log(e m / D)\over m }}
\\
\implies  m\cdot \rademacher(H) &\leq \sqrt{m\cdot {2 D \log(e m / D)}}
\end{align*}
However, in our case $m$ --- the size of the global population $O$ --- is not bounded  (it might be infinite). A more relevant bound is the Massart Lemma \cite[pages 330]{ShalevShwartz2014Understanding}. It says that, if $A$ is a finite set set of $N$ vectors in $\mathbb{R}^m$ whose norm is at most $c$, then:
\begin{align*}
 \rademacher(A)&\leq {c\sqrt{2\cdot \log N} \over m}
\\
\implies  m\cdot \rademacher(A)&\leq c\sqrt{2\cdot \log N}
\end{align*}
In our case, the set $H$ is infinite, but we can look at finite subsets of $H$ that correspond to subsets $T$ with a bounded size. For example, in the proof of Lemma \ref{lem:half-to-half-rnd2} we defined $H^j$ as the set-family containing all the sets in $H$ with cardinality $j$. Equivalently, $H^j$ is the subset of all binary vectors in $H$ with exactly $j$ ones. If $H$ is $d$-bounded, the number of such vectors is at most $(j+1)^{d-1}$, and the norm of these vectors is $\sqrt{j}$, so:
\begin{align*}
& m\cdot \rademacher(H^j)\leq \sqrt{2\cdot j\cdot \log ((j+1)^{d-1})} 
=
\sqrt{2\cdot (d-1)\cdot j\cdot \log (j+1)} 
\end{align*}
The similarity of this expression to our Lemma \ref{lem:half-to-half-rnd2} hints that we could have proved this lemma in a different (possibly more complicated) way using Rademacher complexity and the Massart lemma.

\section{Acknowledgments}
We are grateful to Simcha Haber, Ron Peretz and Gidi Amir for probability-related discussions, and Aryeh Kontorovich and Ryan O'Donnell for learning-related discussions.

\bibliographystyle{apalike}
\bibliography{../erelsegal-halevi}

\end{document}